\documentclass[12pt, twoside, leqno]{article}
\usepackage{amsmath,amsthm}
\usepackage{amssymb,latexsym}
\usepackage{enumerate}
\usepackage{srcltx}

\theoremstyle{plain}
\newtheorem{lem}{Lemma}
\newtheorem{thm}{Theorem}
\newtheorem{cor}{Corollary}

\theoremstyle{remark}
\newtheorem{rem}{Remark}

\numberwithin{equation}{section}

\begin{document}


\title{On R. Chapman's ``evil determinant'': case $p\equiv 1\pmod 4$}

\author{
Maxim Vsemirnov\thanks{This research was supported in part by the
Dynasty Foundation and  RFBR (grant 09-01-00784-a) and the State
Financed Task project no.~6.38.74.2011 at St.~Petersburg State
University.}}

\date{\small St.~Petersburg Department of V.A.Steklov Institute of Mathematics, \\
27 Fontanka, St.~Petersburg, 191023, Russia\\
and \\
St.~Petersburg State University, \\
Department of Mathematics and Mechanics, \\
28 University prospekt, St. Petersburg, 198504,  Russia \\
E-mail: vsemir@pdmi.ras.ru}


\maketitle


\renewcommand{\thefootnote}{}



\renewcommand{\thefootnote}{\arabic{footnote}}
\setcounter{footnote}{0}


\begin{abstract}
For $p\equiv 1\pmod 4$, we prove the formula (conjectured by R. Chapman) for
the determinant of the
$\frac{p+1}{2}\times\frac{p+1}{2}$ matrix $C=(C_{ij})$ with
$C_{ij}=\genfrac{(}{)}{}{1}{j-i}{p}$.
\end{abstract}

\section{Introduction}
Let $p$ be a prime and $\genfrac{(}{)}{}{1}{\cdot}{p}$ denote the
Legendre symbol. Let us set $n=\frac{p-1}{2}$ and consider the
following $(n+1)\times (n+1)$ matrix $C$:
$$
  C_{ij}=\genfrac{(}{)}{}{}{j-i}{p}, \qquad 0\le i,j\le n.
$$
Here and it what follows it is more natural to enumerate rows and
columns starting from 0. R.~Chapman \cite{C3} raised the problem of
evaluating $\det C$; for motivation and related determinants see also
\cite{C1}, \cite{C2}. In particular, Chapman conjectured (see also
\cite[Problem 10]{BCC})  that $\det C$ is always 1 when $p\equiv
3\pmod 4$ and had a conjectural expression for $\det C$ in terms of
the fundamental unit and class number of $\mathbb{Q}(\sqrt{p})$ for
$p\equiv 1\pmod 4$. The sequence $\det C$ for primes $p\equiv 1\pmod
4$ also appears as sequence A179073 in the On-line Encyclopedia of
Integer Sequences \cite{OEIS}.

Chapman's conjecture for $p\equiv 3\pmod 4$ was settled affirmatively
in \cite{V1}. The aim of this paper is to apply methods developed in
\cite{V1} and to evaluate $\det C$ for $p\equiv 1\pmod 4$.

Let $\mathcal{O}$ be the ring of integers of $\mathbb{Q}(\sqrt{p})$.
Let $\varepsilon$ be the fundamental unit in $\mathcal{O}$ and
$h=h(p)$ be the class number.

\begin{thm}
  \label{thm:det1}
Let
\begin{equation}
 \label{eq:defa}
  a+b\sqrt{p} =\left\{
   \begin{array}{l}
      \varepsilon^h, \text{ if } p\equiv 1 \pmod 8, \\
      \varepsilon^{3h}, \text{ if } p\equiv 5 \pmod 8,
   \end{array}
   \right.
\end{equation}
Then $\det C = -a$.
\end{thm}

Our proof is divided into three steps. First, we decompose $C$ into a
product of several matrices; see Theorem~\ref{lem:decomp1} below.
This part resembles a similar step in the evaluation of $\det C$ for
$p\equiv 3\pmod 4$; see \cite{V1} for details. Second, we find
general expressions for certain parametric Cauchy-type determinants
and reduce our problem to a particular case of that calculation.
Finally, we relate the obtained determinants to Dirichlet's  class
number formula for real quadratic fields \cite[Ch.~5, \S4]{BS}, which
involves both class number and the fundamental unit.

Looking at the numerical data, W.~Zudilin and J.~Sondow conjectured
\cite[Entry A179073]{OEIS} that $\det C$ is always negative and even.
This easily follows from our Theorem~1.

\begin{cor}
 \label{cor:a_neg_even}
For $p\equiv 1\pmod 4$, $\det C$ is negative and even.
\end{cor}

\section{Matrix decomposition}

Set $n=(p-1)/2$. Let $\zeta$ be the primitive $p$-th root of unity
with $\arg \zeta = 2\pi/p$. We also fix  its square root
$\zeta^{1/2}$ such that $\arg \zeta^{1/2} = \pi/p$.

Let us consider the following three matrices $U$, $V$, and $D$:
\begin{eqnarray}
  \label{eq:defU}
    U_{ij} &=& \frac{
    \genfrac{(}{)}{}{1}{i}{p} \zeta^{-j-2i} + \genfrac{(}{)}{}{1}{j}{p} \zeta^{-2j-i}
    }
    {\zeta^{-i-j}+\genfrac{(}{)}{}{1}{i}{p} \genfrac{(}{)}{}{1}{j}{p} },
    \quad
    0\le i,j\le n,
    \\
  \nonumber
  V_{ij} &=& \zeta^{2ij},
      \quad
    0\le i,j\le n,
   \\
  \label{eq:defD}
    D_{ii} &=& \prod_{\stackrel{0\le k \le n}{k\neq i}} \frac{1}{\zeta^{2i}-\zeta^{2k}},
      \quad
    0\le i\le n,
  \\
  \label{eq:defD2}
    D_{ij} &=&0,  \quad i\neq j.
\end{eqnarray}
In particular, $V$ is a Vandermonde-type matrix and $D$ is diagonal.
If we set
$$
 g(x)=\prod_{0\le k \le n} (x-\zeta^{2k}),
$$
then the diagonal entries of $D$ can be represented in an alternative
way as
 \begin{equation}
  \label{eq:altD}
  D_{ii}=\frac{1}{g'(\zeta^{2i})}.
 \end{equation}
Finally, let $\tau_p(r)$ be the Gauss sum
 $$
  \tau_p(r) =\sum_{k=1}^{p-1} \genfrac{(}{)}{}{}{kr}{p} \zeta^{k}
    =\sum_{k=1}^{p-1} \genfrac{(}{)}{}{}{k}{p} \zeta^{kr}.
 $$

\begin{thm}
 \label{lem:decomp1} For any prime $p$ such that $p\equiv 1\pmod 4$,
 we have
\begin{equation}
 \label{eq:decomp1}
  C =
  \tau_p(2) \zeta^{(p-1)/4} \cdot VDUDV =
  \genfrac{(}{)}{}{}{2}{p} \sqrt{p} \zeta^{(p-1)/4} \cdot VDUDV.
\end{equation}
\end{thm}

\begin{rem} For $p\equiv 3\pmod 4$ we have (see \cite{V1}) a similar expression
$$
C = -\tau_p(2) \zeta^{-(p+1)/4} \cdot VD\tilde{U}DV,
$$
where
 \begin{equation}
 \label{eq:deftildeU}
\tilde{U}_{ij}= \frac{
    \genfrac{(}{)}{}{1}{i}{p} \zeta^{-j-2i} - \genfrac{(}{)}{}{1}{j}{p} \zeta^{-2j-i}
    }
    {\zeta^{-i-j}-\genfrac{(}{)}{}{1}{i}{p} \genfrac{(}{)}{}{1}{j}{p} }.
 \end{equation}
Both (\ref{eq:defU}) for $p \equiv 1 \pmod{4}$ and
(\ref{eq:deftildeU}) for $p \equiv 3 \pmod{4}$ can be unified as
 $$
 U_{ij}= \frac{
    \genfrac{(}{)}{}{1}{i}{p} \zeta^{-j-2i} + \genfrac{(}{)}{}{1}{-j}{p} \zeta^{-2j-i}
    }
    {\zeta^{-i-j}+\genfrac{(}{)}{}{1}{i}{p} \genfrac{(}{)}{}{1}{-j}{p} }.
 $$
The unified formula for the decomposition becomes   $C =
\genfrac{(}{)}{}{1}{-1}{p} \tau_p(2) \zeta^{(p^2-1)/4} \cdot VDUDV$.
\end{rem}

\begin{proof}[Proof of Theorem~\ref{lem:decomp1}]
Let $B=VDUDV$. We have,
$$
 B_{ij} = \sum_{k=0}^n \sum_{r=0}^n \zeta^{2ki+2rj} \cdot
          \frac{1}{g'(\zeta^{2k})} \cdot
           \frac{1}{g'(\zeta^{2r})} \cdot
           \frac{ \genfrac{(}{)}{}{1}{k}{p}\zeta^{-r-2k} +
           \genfrac{(}{)}{}{1}{r}{p}\zeta^{-2r-k}}
           {\zeta^{-k-r}+\genfrac{(}{)}{}{1}{k}{p}\genfrac{(}{)}{}{1}{r}{p} }.
$$
The term corresponding to $k=r=0$ in the above sum vanishes. The
remaining terms can be arranged into three groups depending on
whether $k=0$, or $r=0$, or $k\neq 0$, $r\neq 0$. More precisely,
$$
  B_{ij}= s_0+s_1+s_2,
$$
where
\begin{eqnarray*}
 s_0 &=&  \frac{1}{g'(1)} \sum_{r=1}^n
          \genfrac{(}{)}{}{}{r}{p}\cdot
           \frac{\zeta^{(2j-1)r}}{g'(\zeta^{2r})},
 \\
 s_1 &=&  \frac{1}{g'(1)} \sum_{k=1}^n
          \genfrac{(}{)}{}{}{k}{p}\cdot
           \frac{\zeta^{(2i-1)k}}{g'(\zeta^{2k})},
 \\
 s_2 &=& \sum_{k=1}^n \sum_{r=1}^n \zeta^{2ki+2rj} \cdot
          \frac{1}{g'(\zeta^{2k})} \cdot
           \frac{1}{g'(\zeta^{2r})} \cdot
           \frac{ \genfrac{(}{)}{}{1}{k}{p}\zeta^{-r-2k} +
           \genfrac{(}{)}{}{1}{r}{p}\zeta^{-2r-k}}
           {\zeta^{-k-r}+\genfrac{(}{)}{}{1}{k}{p}\genfrac{(}{)}{}{1}{r}{p} }.
\end{eqnarray*}
Notice that $\zeta^{-2k-2r}\neq 1$ and
$(\genfrac{(}{)}{}{1}{k}{p}\genfrac{(}{)}{}{1}{r}{p})^2=1$ for $1\le
k,r\le n$. Applying the identity
 $$
  \left(\zeta^a-\genfrac{(}{)}{}{}{k}{p}\genfrac{(}{)}{}{}{r}{p}\right)
  \left(\zeta^a+\genfrac{(}{)}{}{}{k}{p}\genfrac{(}{)}{}{}{r}{p}\right)
  = \zeta^{2a}-1,
 $$
which is valid for $k$, $r$ such that $p\nmid k$, $p\nmid r$, we have
that
\begin{eqnarray*}
 s_2 \!\!&=&\!\! \sum_{k=1}^n \sum_{r=1}^n \frac{\zeta^{2ki+2rj}}
          {g'(\zeta^{2k})g'(\zeta^{2r})} \cdot
           \frac{
            \left(
             \zeta^{-k-r} - \genfrac{(}{)}{}{1}{k}{p}\genfrac{(}{)}{}{1}{r}{p}
            \right)
            \left( \genfrac{(}{)}{}{1}{k}{p}\zeta^{-r-2k} +
           \genfrac{(}{)}{}{1}{r}{p}\zeta^{-2r-k}\right)}
           {\zeta^{-2k-2r}-1}
   \\
    &=&\!\! \sum_{k=1}^n \sum_{r=1}^n \frac{\zeta^{2ki+2rj}}
          {g'(\zeta^{2k})g'(\zeta^{2r})} \cdot
           \frac{
           \genfrac{(}{)}{}{1}{k}{p} (\zeta^{-2r-3k}
           - \zeta^{-2r-k})
           +\genfrac{(}{)}{}{1}{r}{p} (\zeta^{-3r-2k}
           - \zeta^{-r-2k})
           }
           {\zeta^{-2k-2r}-1}
   \\
   &=&\!\! s_3+s_4,
\end{eqnarray*}
where
\begin{eqnarray*}
  s_3 &=& \sum_{k=1}^n \sum_{r=1}^n
     \genfrac{(}{)}{}{}{k}{p}
     \frac{\zeta^{2ki+2rj}}{g'(\zeta^{2k})g'(\zeta^{2r})} \cdot
     \frac{\zeta^{-2r-3k}-\zeta^{-2r-k}}{\zeta^{-2k-2r}-1}
   \\
   &=& \sum_{k=1}^n \sum_{r=1}^n
     \genfrac{(}{)}{}{}{k}{p}
     \frac{\zeta^{2ki+2rj}}{g'(\zeta^{2k})g'(\zeta^{2r})} \cdot
     \frac{\zeta^{-3k}-\zeta^{-k}}{\zeta^{-2k}-\zeta^{2r}},
   \\
  s_4 &=& \sum_{k=1}^n \sum_{r=1}^n
     \genfrac{(}{)}{}{}{r}{p}
     \frac{\zeta^{2ki+2rj}}{g'(\zeta^{2k})g'(\zeta^{2r})} \cdot
     \frac{\zeta^{-3r-2k}-\zeta^{-r-2k}}{\zeta^{-2k-2r}-1}
   \\
     &=& \sum_{k=1}^n \sum_{r=1}^n
     \genfrac{(}{)}{}{}{r}{p}
     \frac{\zeta^{2ki+2rj}}{g'(\zeta^{2k})g'(\zeta^{2r})} \cdot
     \frac{\zeta^{-3r}-\zeta^{-r}}{\zeta^{-2r}-\zeta^{2k}}.
\end{eqnarray*}

By the Lagrange interpolation formula, for any polynomial $f$ of
degree less than $n+1$, we have that
$$
  \frac{f(x)}{g(x)} =
  \sum_{r=0}^{n} \frac{1}{g'(\zeta^{2r})} \cdot \frac{f(\zeta^{2r})}{x-\zeta^{2r}}.
$$
Therefore, for any $x$ different from the roots of $g$,
\begin{equation}
  \label{eq:interp-alt}
  \sum_{r=1}^{n} \frac{1}{g'(\zeta^{2r})} \cdot \frac{f(\zeta^{2r})}{x-\zeta^{2r}}
  =
  \frac{f(x)}{g(x)} -
  \frac{1}{g'(1)} \cdot \frac{f(1)}{x-1}.
\end{equation}
If $k$ runs through $1$,\dots, $n$, then $-2k \pmod p$ runs through
odd integers $p-2$, $p-4$,\dots, $3$, $1$. In particular,
$\zeta^{-2k}$ is not a root of $g$. To evaluate $s_3$ we make
summation with respect to $r$ first and use (\ref{eq:interp-alt}) for
$f(x)=x^{j}$ substituting $x=\zeta^{-2k}$:
\begin{eqnarray*}
 s_3 &=& \sum_{k=1}^n
         \genfrac{(}{)}{}{}{k}{p} \frac{\zeta^{2ki}}{g'(\zeta^{2k})}
         \left(\zeta^{-3k}-\zeta^{-k}\right)
          \biggl(
     \sum_{r=1}^n
     \frac{1}{g'(\zeta^{2r})} \cdot \frac{\zeta^{2rj}}{\zeta^{-2k}-\zeta^{2r}}
      \biggr)
 \\
   &=& \sum_{k=1}^n
         \genfrac{(}{)}{}{}{k}{p} \frac{\zeta^{2ki}}{g'(\zeta^{2k})}
         \left(\zeta^{-3k}-\zeta^{-k}\right)
          \left(
      \frac{\zeta^{-2kj}}{g(\zeta^{-2k})}
       -\frac{1}{g'(1)}\cdot\frac{1}{\zeta^{-2k}-1}
      \right)
  \\
    &=& \sum_{k=1}^n
         \genfrac{(}{)}{}{}{k}{p}
         \frac{\zeta^{2k(i-j)}(\zeta^{-3k}-\zeta^{-k})}
         {g'(\zeta^{2k}) g(\zeta^{-2k})}
      - \frac{1}{g'(1)} \sum_{k=1}^n \genfrac{(}{)}{}{}{k}{p}
      \frac{\zeta^{(2i-1)k}}{g'(\zeta^{2k})}
  \\
    &=& \sum_{k=1}^n
         \genfrac{(}{)}{}{}{k}{p}
         \frac{\zeta^{2k(i-j)}(\zeta^{-3k}-\zeta^{-k})}
         {g'(\zeta^{2k}) g(\zeta^{-2k})}
      - s_1.
\end{eqnarray*}
To evaluate $s_4$ we argue in the same way but now we sum with
respect to $k$ first and substitute $x=\zeta^{-2r}$ into
(\ref{eq:interp-alt}) for $f(x)=x^i$. As a result,
 $$
 s_4 =  \sum_{r=1}^n
         \genfrac{(}{)}{}{}{r}{p}
         \frac{\zeta^{2r(j-i)}(\zeta^{-3r}-\zeta^{-r})}
         {g'(\zeta^{2r}) g(\zeta^{-2r})}
      - s_0.
 $$
Therefore,
\begin{eqnarray*}
 B_{ij} &=& \sum_{k=1}^n
         \genfrac{(}{)}{}{}{k}{p}
         \frac{\zeta^{2k(i-j)}(\zeta^{-3k}-\zeta^{-k})}
         {g'(\zeta^{2k}) g(\zeta^{-2k})}
       + \sum_{r=1}^n
         \genfrac{(}{)}{}{}{r}{p}
         \frac{\zeta^{2r(j-i)}(\zeta^{-3r}-\zeta^{-r})}
         {g'(\zeta^{2r}) g(\zeta^{-2r})}
  \\
      &=&
      \sum_{k=1}^n
         \genfrac{(}{)}{}{}{k}{p}
         \frac{(\zeta^{2k(i-j)}+\zeta^{-2k(i-j)}) (\zeta^{-3k}-\zeta^{-k})}
         {g'(\zeta^{2k}) g(\zeta^{-2k})}.
\end{eqnarray*}
Now we evaluate the denominator of each term. Recall that
$n=(p-1)/2$, so $(-1)^{n+1}=(-1)^{(p+1)/2}=-1$ if $p\equiv 1\pmod 4$.
We have
\begin{eqnarray*}
&& g'(\zeta^{2k}) g(\zeta^{-2k}) =
    \prod_{\stackrel{0\le t \le n}{t\neq k}} (\zeta^{2k}-\zeta^{2t})
    \prod_{0\le t \le n} (\zeta^{-2k}-\zeta^{2t})
    \\
   &&\qquad=
    (-1)^{n+1} (\zeta^{-2k})^{n+1} (\zeta^2)^{(0+1+\cdots+n)}
    \prod_{\stackrel{0\le t \le n}{t\neq k}} (\zeta^{2k}-\zeta^{2t})
    \prod_{0\le t \le n} (\zeta^{2k}-\zeta^{-2t})
    \\
   &&\qquad= -\zeta^{-k(p+1)} \zeta^{(p^2-1)/4} (\zeta^{2k}-1)
       \prod_{\stackrel{0\le t \le p-1}{t\neq k}} (\zeta^{2k}-\zeta^{2t})
    \\
   &&\qquad=
      -\zeta^{-k} \zeta^{(p-1)/4} (\zeta^{2k}-1) ((x^p-1)')_{\vert x=\zeta^{2k}}
    \\
   &&\qquad= -p \zeta^{-3k} \zeta^{(p-1)/4} (\zeta^{2k}-1).
\end{eqnarray*}
Using this together with the fact that $\genfrac{(}{)}{}{}{\cdot}{p}$
is an even character for $p\equiv 1\pmod 4$, we continue as follows:
 \begin{eqnarray*}
    B_{ij}      &=&
       \frac{\zeta^{-(p-1)/4}}{p} \sum_{k=1}^n
         \genfrac{(}{)}{}{}{k}{p}
         (\zeta^{2k(i-j)}+\zeta^{-2k(i-j)})
         \\
     &=&
       \frac{\zeta^{-(p-1)/4}}{p} \sum_{k=1}^n
         \left( \genfrac{(}{)}{}{}{k}{p}
         \zeta^{2k(i-j)} + \genfrac{(}{)}{}{}{-k}{p} \zeta^{-2k(i-j)}\right)
         \\
     &=&
      \frac{\zeta^{-(p-1)/4}}{p} \sum_{k=1}^{p-1}
         \genfrac{(}{)}{}{}{k}{p}
         \zeta^{2k(i-j)}
         \\
     &=&
     \genfrac{(}{)}{}{}{i-j}{p}
     \frac{\zeta^{-(p-1)/4}}{p} \sum_{r=1}^{p-1}
         \genfrac{(}{)}{}{}{r}{p}
         \zeta^{2r}
       =      \genfrac{(}{)}{}{}{j-i}{p}
     \frac{\zeta^{-(p-1)/4} \tau_p(2) }{p}.
 \end{eqnarray*}
Since $\tau_p(2)=\genfrac{(}{)}{}{}{2}{p} \sqrt{p}$ for $p\equiv
1\pmod 4$ (e.g., see \cite[Ch. 6]{IR}), we conclude that
 $$
 \zeta^{(p-1)/4} \tau_p(2) B_{ij}=
 \genfrac{(}{)}{}{}{2}{p} \sqrt{p} \zeta^{(p-1)/4} B_{ij}
 = \genfrac{(}{)}{}{}{j-i}{p} =C_{ij},
 $$
which completes the proof.
\end{proof}

\section{Determinants of Cauchy-like matrices}
The following identity is due to Cauchy \cite{Cauchy}:
 $$
   \det \left( \frac{1}{u_i+v_j} \right)_{i,j=0,\dots, m-1} =
   \prod_{0\le i<j\le m-1} \bigl((u_i-u_j) (v_i-v_j)\bigr)
   \prod_{0\le i,j\le m-1} (u_i+v_j)^{-1}.
 $$
An alternative form of this identity can be obtained by replacing
$v_j$ with $v_j^{-1}$ and multiplying each column by $v_j^{-1}$:
\begin{eqnarray}
  \label{eq:CauchyId1}
  && \det \left( \frac{1}{1+u_iv_j} \right)_{i,j=0,\dots,m-1}
   \\
   \nonumber
  &&\qquad=
   \prod_{0\le i<j\le m-1} \bigl((u_i-u_j) (v_j-v_i)\bigr)
   \prod_{0\le i,j\le m-1} (1+u_iv_j)^{-1}.
\end{eqnarray}
For further connections of (\ref{eq:CauchyId1}) with representation
theory and symmetric functions, see \cite[Ch. 7]{St}.

In this section we evaluate determinants of  several parametric
matrices related to the second form of the Cauchy identity. Let
$\overrightarrow{u} =(u_0,\dots,u_{m-1})$, $\overrightarrow{v}
=(v_0,\dots,v_{m-1})$. Assume further that $1+u_iv_j\neq 0$,
$i,j=0\dots,m-1$. Let $M_{m}(\overrightarrow{u},\overrightarrow{v})$
be the $m\times m$ matrix with
  $$
    (M_m(\overrightarrow{u},\overrightarrow{v}))_{ij} = \frac{u_i+v_j}{1+u_iv_j}.
  $$

\begin{thm}
 \label{lem:detCauchy1}
We have\footnote{When the paper was ready, T.~Amdeberhan informed the
author that an equivalent statement had been proved in \cite{AZ} by a
different method.} $\det M_m(\overrightarrow{u},\overrightarrow{v})=$
  \begin{eqnarray*}
  &&=
          \frac{1}{2} \biggl(
             \prod_{i=0}^{m-1} (1+u_i) \prod_{j=0}^{m-1} (1+v_j)
             +(-1)^m \prod_{i=0}^{m-1} (1-u_i) \prod_{j=0}^{m-1} (1-v_j)
          \biggr)
          \\
     &&\qquad \times
           \prod_{0\le i<j \le m-1} (u_i-u_j)
                \prod_{0 \le i<j \le m-1} (v_j-v_i)
                \prod_{i=0}^{m-1} \prod_{j=0}^{m-1} (1+u_iv_j)^{-1}.
 \end{eqnarray*}
\end{thm}

\begin{proof}
Let $J$ be the $m\times m$ matrix with all entries equal to 1.
Consider
 $$
   f(t) =\det(tJ+M_m(\overrightarrow{u},\overrightarrow{v})).
 $$
Since $J$ has rank 1, there are two invertible matrices $H_1$ and
$H_2$ with coefficients independent of the variable $t$, such that
$$
  H_1 J H_2 =
  \begin{pmatrix}
   1 & 0 & \dots & 0 \\
   0 & 0 & \dots & 0 \\
   \vdots & \vdots & \ddots & \vdots \\
   0 & 0 & \dots & 0
  \end{pmatrix}.
$$
Since $f(t) = \det{H_1}^{-1}  \det (t H_1 J H_2 + H_1
M_m(\overrightarrow{u},\overrightarrow{v}) H_2) \det{H_2}^{-1}$, the
function $f$ is linear with respect to $t$. In particular, $\det M_m
(\overrightarrow{u},\overrightarrow{v})= f(0)= (f(1)+f(-1))/2$. On
the other hand,
\begin{eqnarray*}
 f(1) \!&=&\! \det \left( 1+ \frac{u_i+v_j}{1+u_i v_j} \right)_{i,j=0,\dots m-1}
     =   \det \left(\frac{(1+u_i)(1+v_j)}{1+u_i v_j} \right)_{i,j=0,\dots,m-1}
     \\
     \!&=&\!  \prod_{i=0}^{m-1} (1+u_i)
        \prod_{j=0}^{m-1} (1+v_j)
        \cdot \det \left(\frac{1}{1+u_i v_j} \right)_{i,j=0,\dots,m-1},
\end{eqnarray*}
In a similar way
\begin{eqnarray*}
 f(-1) &=& \det \left( -1+ \frac{u_i+v_j}{1+u_i v_j} \right)_{i,j=0,\dots m-1}
      \\
     &=&   \det \left(\frac{-(1-u_i)(1-v_j)}{1+u_i v_j} \right)_{i,j=0,\dots,m-1}
     \\
     &=&  (-1)^m \prod_{i=0}^{m-1} (1-u_i)
        \prod_{j=0}^{m-1} (1-v_j)
        \cdot \det \left(\frac{1}{1+u_i v_j} \right)_{i,j=0,\dots,m-1}.
\end{eqnarray*}
Combining this with (\ref{eq:CauchyId1}) we complete the proof.
\end{proof}

Now let $\overrightarrow{x} =(x_1,\dots,x_{m})$, $\overrightarrow{y}
=(y_1,\dots,y_{m})$ and assume that $1+x_iy_j\neq 0$, $1+x_i\neq 0$,
$1+y_j\neq 0$, $i,j=1,\dots,m$. Let
$W_m(\overrightarrow{x},\overrightarrow{y})$ be the following
$(m+1)\times(m+1)$ matrix:
$$
  W_m(\overrightarrow{x},\overrightarrow{y}) =
  \left(
  \begin{array}{c|ccc}
     0 & 1 & \dots & 1 \\
    \hline
     1 &   &       & \\
     \vdots &  &  M_m(\overrightarrow{x}, \overrightarrow{y}) & \\
     1 &   &  &
  \end{array}
  \right).
$$

\begin{thm}
  \label{lem:detCauchy2}
  We have    $\det W_m(\overrightarrow{x},\overrightarrow{y})=$
  \begin{multline*}
     =
         - \frac{1}{2} \biggl(
             \prod_{i=1}^{m} (1+x_i) \prod_{j=1}^{m} (1+y_j)
             -(-1)^m \prod_{i=1}^{m} (1-x_i) \prod_{j=1}^{m} (1-y_j)
          \biggr)
          \\
          \times
           \prod_{1\le i<j \le m} (x_i-x_j)
                \prod_{1\le i<j \le m} (y_j-y_i)
                \prod_{i=1}^{m} \prod_{j=1}^{m} (1+x_iy_j)^{-1}.
 \end{multline*}
\end{thm}

\begin{proof}
Since
$$
  M_{m+1}((1,\overrightarrow{x}),(1,\overrightarrow{y})) =
  \left(
  \begin{array}{c|ccc}
     1 & 1 & \dots & 1 \\
    \hline
     1 &   &       & \\
     \vdots &  &  M_m(\overrightarrow{x}, \overrightarrow{y}) & \\
     1 &   &  &
  \end{array}
  \right),
$$
we have that
\begin{eqnarray}
  \label{eq:detW}
 &&\det  W_{m}(\overrightarrow{x},\overrightarrow{y})
  \\
 \nonumber
 && \qquad =
  \det M_{m+1}((1,\overrightarrow{x}),(1,\overrightarrow{y}))
   - \det
  \left(
  \begin{array}{c|ccc}
     1 & 0 & \dots & 0 \\
    \hline
     1 &   &       & \\
     \vdots &  &  M_m(\overrightarrow{x}, \overrightarrow{y}) & \\
     1 &   &  &
  \end{array}
  \right)
  \\
  \nonumber
  && \qquad =
  \det M_{m+1}((1,\overrightarrow{x}),(1,\overrightarrow{y})) -
  \det M_{m}(\overrightarrow{x},\overrightarrow{y}).
\end{eqnarray}
By Theorem~\ref{lem:detCauchy1},
 \begin{eqnarray*}
  &&\det M_{m+1}((1,\overrightarrow{x}),(1,\overrightarrow{y}))
  = \frac{(1+1)(1+1)}{2(1+1)} \prod_{i=1}^m (1+x_i) \prod_{j=1}^m (1+y_j)
   \\
   && \qquad
   \times
    \prod_{j=1}^m (1-x_j)
    \prod_{1\le i<j \le m} (x_i-x_j)
    \prod_{j=1}^m (y_j-1)
    \prod_{1\le i<j \le m} (y_j-y_i)
   \\
   && \qquad
   \times
    \prod_{j=1}^m (1+y_j)^{-1} \prod_{i=1}^m (1+x_i)^{-1}
    \prod_{i=1}^m \prod_{j=1}^m (1+x_i y_j)^{-1} +0
   \\
   && =
    (-1)^m
    \prod_{i=1}^m (1-x_i)
    \prod_{j=1}^m (1-y_j)
    \prod_{1\le i<j \le m} (x_i-x_j)
    \prod_{1\le i<j \le m} (y_j-y_i)
   \\
   && \qquad \times
    \prod_{i=1}^m \prod_{j=1}^m (1+x_i y_j)^{-1}.
 \end{eqnarray*}
Applying Theorem~\ref{lem:detCauchy1} to the evaluation of $\det
M_{m}(\overrightarrow{x},\overrightarrow{y})$ and using
(\ref{eq:detW}) we complete the proof.
\end{proof}

\section{Evaluation of the determinant}

Recall that $n=(p-1)/2$. Let $G$ be the diagonal matrix with
 \begin{equation}
  \label{eq:defG}
 G_{00}=1,\quad G_{ii}=\genfrac{(}{)}{}{}{i}{p}\zeta^i,\quad i=1,\dots,n.
  \end{equation}
Set
 \begin{equation}
 \label{eq:defW}
 W=GUG.
 \end{equation}
By a direct computation, $W_{00}=0$,
\begin{eqnarray*}
  W_{0j} &=& W_{j0}=1, \quad j=1,\dots,n,
  \\
  W_{ij} &=& \frac{\genfrac{(}{)}{}{}{i}{p}\zeta^i+\genfrac{(}{)}{}{}{j}{p}\zeta^j}
                  {1+\genfrac{(}{)}{}{}{i}{p}\genfrac{(}{)}{}{}{j}{p}\zeta^{i+j}}
  , \quad i,j=1,\dots,n.
\end{eqnarray*}
In particular, $W=W_n(\overrightarrow{x},\overrightarrow{y})$, where
$x_i=y_i=\genfrac{(}{)}{}{1}{i}{p}\zeta^i$, $i=1,\dots, n$. Now we
apply Theorem~\ref{lem:detCauchy2}. The last product in the
expression for $\det W_n$ can be transformed in the following way:
first we extract terms with $i=j$ and then we join together the two
factors corresponding to $(i,j)$ and $(j,i)$ for $i\neq j$. Taking
into account that $(-1)^n=(-1)^{(p-1)/2}=1$ for $p\equiv1\pmod 4$, we
conclude that
\begin{multline}
  \det W =
         -\frac{(-1)^{n(n-1)/2}}{2} \biggl(
             \prod_{j=1}^{n} \left(1 +\genfrac{(}{)}{}{}{j}{p} \zeta^j\right)^2
             - \prod_{j=1}^{n}
              \left(1 -\genfrac{(}{)}{}{}{j}{p} \zeta^j\right)^2
          \biggr)
          \\
  \label{eq:detW1}
          \times\!\!
           \prod_{1\le i<j \le n} \!
              \left (\genfrac{(}{)}{}{}{i}{p} \zeta^i-
              \genfrac{(}{)}{}{}{j}{p} \zeta^j\right)^2
           \prod_{1\le i<j \le n}
              \left (1+\genfrac{(}{)}{}{}{i}{p}
              \genfrac{(}{)}{}{}{j}{p} \zeta^{i+j}\right)^{-2}
                \prod_{j=1}^{n}  (1+\zeta^{2j})^{-1}.
\end{multline}

The following auxiliary result is an easy consequence of standard
methods of evaluating Gauss sums. We present its proof for the sake
of completeness.

\begin{lem}
  \label{lem:Gauss}
 If $p\nmid r$, then
   \begin{equation}
     \label{eq:Gauss1}
   \prod_{j=1}^n (\zeta^{2rj}-\zeta^{-2rj}) = \genfrac{(}{)}{}{}{r}{p} \sqrt{p}.
   \end{equation}
\end{lem}

\begin{proof}
 By \cite[Proposition 6.4.3]{IR},
 \begin{eqnarray*}
   \sqrt{p} &=& \prod_{k=1}^n (\zeta^{2k-1}-\zeta^{-(2k-1)})
     = (-1)^n \prod_{k=1}^n (\zeta^{p-(2k-1)}-\zeta^{-p+(2k-1)})
   \\
     &=& \prod_{j=1}^n (\zeta^{2j}-\zeta^{-2j}).
 \end{eqnarray*}
Let us apply the automorphism $\gamma$ of $\mathbb{Q}(\zeta)$ induced
by $\gamma(\zeta)=\zeta^r$. It follows from the standard properties
of Gauss sums that
$\gamma(\sqrt{p})=\genfrac{(}{)}{}{1}{r}{p}\sqrt{p}$, which completes
the proof.
\end{proof}

\begin{cor}
  \label{cor:Gauss}
  We have
\begin{eqnarray}
   \label{eq:Gauss2}
   && \prod_{j=1}^n (\zeta^{j/2}-\zeta^{-j/2}) = \genfrac{(}{)}{}{}{2}{p} \sqrt{p}
   =(-1)^{n/2} \sqrt{p},
   \\
   \label{eq:Gauss3}
   && \prod_{j=1}^n (1+\zeta^{2j}) = \zeta^{n(n+1)/2} \genfrac{(}{)}{}{}{2}{p}.
\end{eqnarray}
\end{cor}

\begin{proof}
We have $\zeta^{1/2} =- \zeta^{(p+1)/2}$. Hence,
 \begin{eqnarray*}
  \prod_{j=1}^n (\zeta^{j/2}-\zeta^{-j/2}) &=&
  \prod_{j=1}^n ((-\zeta^{(p+1)/2})^j-(-\zeta^{(p+1)/2})^{-j} )
  \\
  &=&
  (-1)^{n/2}   \prod_{j=1}^n ((\zeta^{(p+1)/2})^j-(\zeta^{(p+1)/2})^{-j} ).
 \end{eqnarray*}
(Here we extract $-1$ from each term that corresponds to an odd $j$.)
The  first desired identity now follows from Lemma~\ref{lem:Gauss}
applied to $r=(1-p)/4$ and from the fact that, for $p\equiv 1\pmod
4$, $(-1)^{n/2}=(-1)^{(p-1)/4} = \genfrac{(}{)}{}{1}{2}{p}$. To prove
the second identity we use the equality $1+\zeta^{2j}=\zeta^j
(\zeta^{2j}-\zeta^{-2j})(\zeta^{j}-\zeta^{-j})^{-1}$ and
Lemma~\ref{lem:Gauss} applied to $r=1$ and $r=(p+1)/2$.
\end{proof}

\begin{lem}
  \label{lem:spec_fact}
  We have
 $$ \frac{1}{2}
    \biggl(
    \prod_{j=1}^n \left(1+\genfrac{(}{)}{}{}{j}{p}\zeta^j\right)^2
          - \prod_{j=1}^n \left(1-\genfrac{(}{)}{}{}{j}{p}\zeta^j\right)^2
    \biggr)
    = (-1)^{n/2} \zeta^{n(n+1)/2}  a \sqrt{p},
 $$
where $a$ is defined in {\rm (\ref{eq:defa})}.
\end{lem}

\begin{proof} Let
$$
  s=\prod_{j=1}^n \left(1+\genfrac{(}{)}{}{}{j}{p}\zeta^j\right)^2.
$$
Notice that $(1+\genfrac{(}{)}{}{1}{j}{p}\zeta^j)^2=
(\zeta^j+\genfrac{(}{)}{}{1}{j}{p})^2$. We have
\begin{eqnarray*}
  s&=& \prod_{j=1}^n \left(\zeta^j+\genfrac{(}{)}{}{}{j}{p}\right)^2
    = \zeta^{n(n+1)/2}
    \prod_{j=1}^n \left(\zeta^{j/2}+\genfrac{(}{)}{}{}{j}{p}\zeta^{-j/2}\right)^2
      \\
   &=&
    \zeta^{n(n+1)/2}
    \prod_{\stackrel{1\le j \le n}{(j/p)=-1}} (\zeta^{j/2}-\zeta^{-j/2})^2
    \prod_{\stackrel{1\le j \le n}{(j/p)=1}} (\zeta^{j/2}+\zeta^{-j/2})^2
      \\
   &=&
    \zeta^{n(n+1)/2}
    \prod_{\stackrel{1\le j \le n}{(j/p)=-1}} (\zeta^{j/2}-\zeta^{-j/2})^2
    \prod_{\stackrel{1\le j \le n}{(j/p)=1}} (\zeta^{j/2}-\zeta^{-j/2})^{-2}
  \\
   &&\qquad
    \times
    \prod_{\stackrel{1\le j \le n}{(j/p)=1}} (\zeta^{j}-\zeta^{-j})^{2}
      \prod_{j=1}^n (\zeta^{j/2}-\zeta^{-j/2})^{-1}
      \prod_{j=1}^n (\zeta^{j/2}-\zeta^{-j/2}).
 \end{eqnarray*}
Since $\genfrac{(}{)}{}{1}{\cdot}{p}$ is an even character for
$p\equiv 1\pmod 4$, the number of quadratic residues modulo $p$ on
the interval $[1,n]$ equals the number of quadratic non-residues on
the same interval. Therefore, we can continue as follows
\begin{eqnarray*}
  s &=& \zeta^{n(n+1)/2}
    \prod_{\stackrel{1\le j \le n}{(j/p)=-1}} \left(\sin\frac{\pi j}{p}\right)^2
    \prod_{\stackrel{1\le j \le n}{(j/p)=1}} \left(\sin\frac{\pi j}{p}\right)^{-2}
    \prod_{\stackrel{1\le j \le n}{(j/p)=1}} \left(\sin\frac{2\pi j}{p}\right)^{2}
   \\
   && \qquad \times
      \prod_{j=1}^n \left(\sin\frac{\pi j}{p}\right)^{-1}
      \prod_{j=1}^n (\zeta^{j/2}-\zeta^{-j/2}).
\end{eqnarray*}
Let us consider the third and the fourth products more carefully. We
have $\sin(2\pi j/p)=\sin(\pi(p-2 j)/p)$. Moreover, the map $j\mapsto
p-2j$  gives a bijection between the sets
 $$
 \{ j : n/2<j\le n,\ \genfrac{(}{)}{}{1}{j}{p}=1 \} \text{ and }
 \{ k : 1\le k\le n,\ k \text{ is odd, } \genfrac{(}{)}{}{1}{k}{p}=
 \genfrac{(}{)}{}{1}{2}{p} \},
 $$
while the map $j\mapsto 2j$ is the bijection between
 $$
 \{ j : 1\le j\le n/2,\ \genfrac{(}{)}{}{1}{j}{p}=1 \} \text{ and }
 \{ k : 1\le k\le n,\ k \text{ is even, } \genfrac{(}{)}{}{1}{k}{p}=
 \genfrac{(}{)}{}{1}{2}{p} \}.
 $$
Therefore,
 $$
    \prod_{\stackrel{1\le j \le n}{(j/p)=1}} \left(\sin\frac{2\pi j}{p}\right)^{2}
      \prod_{j=1}^n \left(\sin\frac{\pi j}{p}\right)^{-1}
   =\!
    \prod_{\stackrel{1\le k \le n}{(k/p)=(2/p)}}\!\!
    \sin\frac{\pi k}{p}
    \prod_{\stackrel{1\le k \le n}{(k/p)=-(2/p)}}\!
    \left(\sin\frac{\pi k}{p}\right)^{-1}.
 $$
By Dirichet's class number formula for real quadratic fields (e.g.,
see \cite[Ch.~5, \S4]{BS}),
 $$
  \varepsilon^h=
    \prod_{\stackrel{1\le j \le n}{(j/p)=-1}} \sin\frac{\pi j}{p}
    \prod_{\stackrel{1\le j \le n}{(j/p)=1}} \left(\sin\frac{\pi j}{p}\right)^{-1}.
 $$
Combining with the above equalities, we conclude that
\begin{equation}
 \label{eq:prod1}
  \prod_{j=1}^n \left(1+\genfrac{(}{)}{}{}{j}{p}\zeta^j\right)^2
  = \zeta^{n(n+1)/2} \varepsilon^{(2-(2/p))h}
    \prod_{j=1}^n (\zeta^{j/2}-\zeta^{-j/2}).
\end{equation}
In a similar way,
\begin{equation}
 \label{eq:prod2}
  \prod_{j=1}^n \left(1-\genfrac{(}{)}{}{}{j}{p}\zeta^j\right)^2
  = \zeta^{n(n+1)/2} \varepsilon^{-(2-(2/p))h}
    \prod_{j=1}^n (\zeta^{j/2}-\zeta^{-j/2}).
\end{equation}

It is well known that $\varepsilon^h$ is a quadratic unit of norm
$-1$ (which is equivalent to the fact that the norm of $\varepsilon$
is $-1$ and $h$ is odd); e.g. see \cite[Ch.~4, \S18.4]{Ha} or
\cite[Ch.~5, Sec.~4, Ex.~5]{BS}. It follows that
 \begin{equation}
  \label{eq:prod3}
  \varepsilon^{(2-(2/p))h}-\varepsilon^{-(2-(2/p))h}=2a,
 \end{equation}
where $a$ is defined in (\ref{eq:defa}).

Applying Corollary~\ref{cor:Gauss} to the evaluation of
$\prod_{j=1}^n (\zeta^{j/2}-\zeta^{-j/2})$ we complete the proof.
\end{proof}

\begin{proof}[Proof of Theorem~\ref{thm:det1}]
It follows from Theorem~\ref{lem:decomp1}, Lemma~\ref{lem:spec_fact}
and equations (\ref{eq:defW}), (\ref{eq:detW1}), (\ref{eq:Gauss3})
that
\begin{multline}
  \label{eq:detC2}
  \det C = -\zeta^{(n+1)(p-1)/4}
     \left( \genfrac{(}{)}{}{}{2}{p}\sqrt{p} \right)^{n+2}
    \\
     \times (\det V)^2 (\det D)^2 (\det G)^{-2} f_1^2 f_2^{-2} a,
\end{multline}
where
\begin{eqnarray*}
 f_1 &=&           \prod_{1\le i<j \le n}
              \left (\genfrac{(}{)}{}{}{i}{p} \zeta^i-
              \genfrac{(}{)}{}{}{j}{p} \zeta^j\right),
\\
 f_2 &=&
           \prod_{1\le i<j \le n}
              \left (1+\genfrac{(}{)}{}{}{i}{p}
              \genfrac{(}{)}{}{}{j}{p} \zeta^{i+j}\right)
\end{eqnarray*}
and $a$ is defined by (\ref{eq:defa}). Clearly,  $\det G\in (\mathbb{Z}[\zeta])^\ast$ by (\ref{eq:defG}).

Let us recall some well-known facts about the arithmetic of the ring
$\mathbb{Z}[\zeta]$. The reader may find further details in \cite[Ch.
13]{IR}. The ideal $(1-\zeta)$ is prime in $\mathbb{Z}[\zeta]$ and
$p=\alpha (1-\zeta)^{p-1}$, where $\alpha$ is a unit in
$\mathbb{Z}[\zeta]$. In particular, $\sqrt{p}=\tau_p(1)=\alpha_1
(1-\zeta)^{n}$, where $\alpha_1 \in (\mathbb{Z}[\zeta])^\ast$.
Finally, $(1-\zeta^c)/(1-\zeta)$ and $1+\zeta^c$ are in
$(\mathbb{Z}[\zeta])^\ast$ provided $p\nmid c$.

Notice that
 $$
  \det V = \prod_{0\le i<j\le n} (\zeta^{2j}-\zeta^{2i})
 $$
by the well-known evaluation of the Vandermonde determinant. Using
this together with the definition of $D$ (see (\ref{eq:defD}),
(\ref{eq:defD2})) and the above observations we find that
 $$
 (\sqrt{p})^{n+1}  (\det V)^2 (\det D)^2 \in
 (\mathbb{Z}[\zeta])^\ast.
 $$

Since $\genfrac{(}{)}{}{1}{\cdot}{p}$ is an even character and
$n=(p-1)/2$, there are $n/2$ quadratic residues and $n/2$ quadratic
non-residues modulo $p$ on the interval $[1,n]$. Thus,
\begin{eqnarray*}
  && \#\{ (i,j) : 1\le i,j\le n,\
   \genfrac{(}{)}{}{1}{i}{p}=\genfrac{(}{)}{}{1}{j}{p}
   \}  = n^2/2,
  \\
  && \#\{ (i,j) : 1\le i,j\le n,\
   \genfrac{(}{)}{}{1}{i}{p}=-\genfrac{(}{)}{}{1}{j}{p}
   \}  = n^2/2.
\end{eqnarray*}
Since the pairs $(i,i)$ are in the first set and the pairs $(i,j)$
and $(j,i)$ are in one and the same set,
\begin{eqnarray*}
  && \#\{ (i,j) : 1\le i<j\le n,\
   \genfrac{(}{)}{}{1}{i}{p}=\genfrac{(}{)}{}{1}{j}{p}
   \}  = n(n-2)/4,
  \\
  && \#\{ (i,j) : 1\le i<j\le n,\
   \genfrac{(}{)}{}{1}{i}{p}=-\genfrac{(}{)}{}{1}{j}{p}
   \}  = n^2/4.
\end{eqnarray*}
In addition,
\begin{eqnarray*}
   \genfrac{(}{)}{}{1}{i}{p} \zeta^i-   \genfrac{(}{)}{}{1}{j}{p}\zeta^j & \in &
   \left\{
   \begin{array}{l}
     (\mathbb{Z}[\zeta])^\ast, \mbox{ if }   \genfrac{(}{)}{}{1}{i}{p}\neq \genfrac{(}{)}{}{1}{j}{p}, \\
     (1-\zeta) (\mathbb{Z}[\zeta])^\ast, \mbox{ if }   \genfrac{(}{)}{}{1}{i}{p}=\genfrac{(}{)}{}{1}{j}{p},
   \end{array}
   \right.
\\
   1+\genfrac{(}{)}{}{1}{i}{p} \genfrac{(}{)}{}{1}{j}{p}\zeta^{i+j} & \in &
   \left\{
   \begin{array}{l}
     (\mathbb{Z}[\zeta])^\ast, \mbox{ if }   \genfrac{(}{)}{}{1}{i}{p}= \genfrac{(}{)}{}{1}{j}{p}, \\
     (1-\zeta) (\mathbb{Z}[\zeta])^\ast, \mbox{ if }   \genfrac{(}{)}{}{1}{i}{p}\neq \genfrac{(}{)}{}{1}{j}{p}.
   \end{array}
   \right.
\end{eqnarray*}
Therefore, $\sqrt{p} f_1^2f_2^{-2} \in (\mathbb{Z}[\zeta])^\ast$.

Finally, notice that $\zeta$ is of odd multiplicative order and
therefore any power of $\zeta$ is a square in $\mathbb{Z}[\zeta]$.
Using (\ref{eq:detC2}) we conclude that $\det C= -a \delta^2$, where
$\delta\in (\mathbb{Z}[\zeta])^\ast$. Since $\det C$ is an integer
and $a$ is a non-zero integer or half-integer, we have that
$\delta^2\in \mathbb{Q}$. Hence, $\delta^2\in\mathbb{Z}^\ast$, i.e.,
$\delta^2=\pm 1$. On the other hand, $\mathbb{Q}(\zeta)$ does not
contain primitive fourth roots of 1, since
$[\mathbb{Q}(\zeta,\sqrt[4]{1}):\mathbb{Q}]=
[\mathbb{Q}(\zeta\cdot\sqrt[4]{1}):\mathbb{Q}]=2(p-1)$ and
$[\mathbb{Q}(\zeta):\mathbb{Q}]=p-1$. Therefore, $\delta^2=1$. This
completes the proof.
\end{proof}

\begin{proof}[Proof of Corollary~\ref{cor:a_neg_even}]
Since $\varepsilon>1$, we have immediately that $a>0$. Hence, $\det
C<0$.

Now write $\varepsilon^h = (\alpha + \beta \sqrt{p})/2$, where
$\alpha$ and $\beta$ are integers. Since $\varepsilon^h$ is a
quadratic unit of norm $-1$ (see \cite[Ch.~4, \S18.4]{Ha} or
\cite[Ch.~5, Sec.~4, Ex.~5]{BS}), we have
\begin{equation}
  \label{eq:normeps}
     \alpha^2-p\beta^2=-4.
\end{equation}
Consider (\ref{eq:normeps}) modulo 16. A direct search shows that
\begin{itemize}
\item for $p \equiv 1, 9 \pmod{16}$,  $\alpha \equiv 0$, 4, 8, or
    12 $ \pmod{16}$, i.e., $a=\alpha/2$ is even;
\item for $p \equiv  5, 13 \pmod{16}$, $\alpha^3+3\alpha\beta^2 p
    \equiv 0 \pmod{16}$.  In particular, raising $\varepsilon^h$
    to the third power, we have that  $a=(\alpha^3+3\alpha\beta^2
    p)/8$ is even.
\end{itemize}
\end{proof}

\section*{Acknowledgements} The author is grateful to J.~Sondow and
T.~Amdeberhan for commenting an earlier version of the paper.

\end{document}